\documentclass[12pt]{amsart}

\usepackage{a4wide}
\usepackage{amsmath}
\usepackage{amsfonts}
\usepackage{amssymb}
\usepackage{epsf}
\usepackage{graphicx}
\usepackage{colonequals}
\usepackage{color}
\usepackage{algorithm,algorithmic,color,amsmath}
\usepackage{tabularx}
\usepackage[percent]{overpic}

\newtheorem{theorem}{Theorem}[section]

\newtheorem{proposition}[theorem]{Proposition}
\newtheorem{corollary}[theorem]{Corollary}

\theoremstyle{definition}

\newtheorem{conjecture}[theorem]{Conjecture}
\theoremstyle{remark}

\numberwithin{equation}{section}

\newcommand{\ce}{\:\colonequals\:}

\newcommand{\floor}[1]{\left \lfloor #1 \right \rfloor}

\newcommand{\Z}{\mathbb{Z}}

\newcommand{\R}{\mathbb{R}}

\newcommand {\EE}  {{\mathcal E}}

\title[A number theoretic problem on polynomials with bounded roots]{A number theoretic problem on the distribution of polynomials with bounded roots}

\author [P.~Kirschenhofer and M.~Weitzer]{Peter~Kirschenhofer and Mario~Weitzer}
\address{Chair of Mathematics and Statistics,
Montanuniversitaet Leoben, Franz Josef-Stra\ss{}e 18, A-8700 Leoben,
AUSTRIA} \email{Peter.Kirschenhofer@unileoben.ac.at,
Mario.Weitzer@unileoben.ac.at}

\thanks{Both authors are supported by the Austrian Science Fund (FWF) Doctoral Program  W1230 ``Discrete Mathematics'', the first author is also supported by the Franco-Austrian research project I1136 granted by the French National Research Agency (ANR) and the FWF. The second author would like to express his gratitude to Prof.~Attila Peth\H{o} and the Departments of Computer Science and Mathematics at the University of Debrecen for their warm hospitality during his stay there in Fall 2013, when his attention was drawn to this problem.}

\date{\today}



\begin{document}
\begin{abstract}
Let $\EE_d^{(s)}$ denote the set of coefficient vectors
$(a_1,\dots,a_d)\in  \R^d$ of contractive polynomials
$x^d+a_1x^{d-1}+\dots+a_d\in \R[x]$ that have exactly $s$ pairs of complex conjugate
roots and let $v_d^{(s)}=\lambda_d(\EE_d^{(s)})$ be its ($d$-dimensional) Lebesgue
measure. We settle the instance $s=1$ of a conjecture by Akiyama and
Peth\H{o}, stating that the ratio $v_d^{(s)}/v_d^{(0)}$ is an integer
for all $d\ge 2s.$ Moreover we establish the surprisingly simple
formula $v_d^{(1)}/v_d^{(0)} = (P_d(3)-2d-1)/4,$ where $P_d(x)$ are
the Legendre polynomials.
\end{abstract}
\maketitle \setcounter{section}{0}

\section{Introduction}

Let $\EE_d$ denote the set of all coefficient vectors
$(a_1,\dots,a_d)\in  \R^d$ of polynomials
$x^d+a_1x^{d-1}+\dots+a_d$ with coefficients in $\R$ and all
roots having absolute value less than 1, and let
$\EE_d^{(s)}$ denote the subset of the coefficient vectors of those
polynomials in $\EE_d$ that have exactly $s$ pairs of complex
conjugate roots. Let furthermore $v_d=\lambda_d (\EE_d)$ and
$v_d^{(s)}=\lambda_d(\EE_d^{(s)})$ denote the $d$-dimensional Lebesgue measures of the
referring sets.

The sets $\EE_d$ have been studied by several authors in different context, compare e.g. Schur~ \cite{Schur:18},
Fam and Meditch \cite{Fam-Meditch:78} or Fam \cite{Fam:89}. More recently, the regions $\EE_d$  have become of interest in the study of ``shift radix systems'', since the regions where those systems have a certain periodicity property are in close connection with the regions  $\EE_d$  (compare e.g. Kirschenhofer et al. \cite{KPST}) . Fam \cite{Fam:89} established the formula
\begin{equation}
\begin{split}
v_d =
\begin{cases}
2^{2m^2}\prod_{j=1}^m \frac {(j-1)!^4}{(2j-1)!^2} \qquad &\text {if } d=2m,
\\
2^{2m^2+2m+1}\prod_{j=1}^m \frac {j!^2(j-1)!^2}{(2j-1)!(2j+1)!}
\qquad &\text{if } d=2m+1.
\end{cases}\\
\end{split}
\end{equation}
 In \cite{API} Akiyama and
Peth\H{o} gave a number of results on the quantities $v_d^{(s)},$
including an
integral representation for general $s$ from which they derived an explicit formula in the instance $s=0$ as well as a somewhat involved expression for  $s=1$ reading
\begin{equation}\label{vd}
\begin{split}
v_d^{(0)} =& \frac{2^{d(d + 1)/2}}{d!}S_d(1,1,1/2),\\
v_d^{(1)} =& 2^{(d - 1)(d - 2)/2-2}
\sum_{j=0}^{d-2}\sum_{k=0}^{d-2-j}
\frac{(-1)^{d-k}2^{2d-2-2k-j}}{j!k!(d-2-j-k)!}
B_{d-2}(d-2-k,d-2-k-j)\\
&\int_{z=0}^{1}\int_{y=-2\sqrt{z}}^{2\sqrt{z}} y^j(y+z+1)^k\:dy\:dz
\end{split}
\end{equation}
for $d \geq 2$ and $0 \leq k \leq j \leq d$ where
\begin{equation}
\begin{split}
S_d(1,1,1/2) \ce& \frac{1}{\prod_{i=0}^{d-1} \binom{2i+1}{i}}\\
\end{split}
\end{equation}
is a special instance of the Selberg integral $S_n(\alpha,\beta,\gamma)$ and where
\begin{equation}
\begin{split}
B_d(j,k) \ce& \prod_{i=1}^{k}\frac{2 + (d - i - 1)/2}{3 + (2d - i -
1)/2} \frac{ \prod_{i=1}^{j}(1 + (d - i)/2) \prod_{i=1}^{k}(1 + (d -
i)/2) }{ \prod_{i=1}^{j+k}(2 + (2d - i - 1)/2) }
S_d(1,1,1/2).\\
\end{split}
\end{equation}
is a special instance of Aomoto's generalization of the Selberg integral (compare Andrews at al. \cite[Section 8]{AndrewsAskeyRoy:99} for Selberg's and Aomoto's integrals).

Furthermore, Akiyama and
Peth\H{o} in \cite{API} proved that the ratios $v_d^{(s)}/v_d^{(0)}$ are rational, and, motivated by extensive numerical evidence,   stated the following
\begin{conjecture} \label{conjecture} \cite[Conjecture 5.1]{API}
The quotient
$$v_d^{(s)}/v_d^{(0)}$$
is an integer for all non-negative integers $d,s$ with $d\ge 2s.$
\end{conjecture}
In Section 2 of this paper we will prove this conjecture for
the instance $s=1$ and in addition give a surprisingly simple
explicit formula for the quotient in this case involving the
Legendre polynomials evaluated at $x=3$. In the proof we will
combine several transformations of binomial sums,
one of them corresponding to a special instance of Pfaff's
reflection law for hypergeometric functions. We refer the reader in
particular to the standard reference \cite[Section 5]{Knuth:94} for
the techniques that we will apply.

In Section 3 we will use our main theorem to establish a
linear recurrence for the sequence $\left(v_d^{(1)}/{v_d^{(0)}}\right)_{d\ge 0},$
and from its generating function will derive its asymptotic behaviour for
$d \to \infty.$ Combined with a result from \cite{API}, this also gives information
on the asymptotic behaviour of the probability $p_d^{(1)}= {v_d^{(1)}}/{v_d}$ of a
contractive polynomial of degree $d$ to have exactly one pair of complex conjugate roots.

In the final section we discuss possible generalizations of our results.

\section{Main result}
\begin{theorem} \label{mainresult} The quotient ${v_d^{(1)}}/{v_d^{(0)}}$  is an
integer for each $d\geq2$. Furthermore we have
\begin{equation*}
\begin{split}
\frac{v_d^{(1)}}{v_d^{(0)}} &= \frac{P_d(3)-2d-1}{4}, \qquad \text{where}\\
P_d(x) &\ce 2^{-d}\sum_{k=0}^{\floor {d/2}}(-1)^k \binom{d-k}{k} \binom{2d-2k}{d-k}
x^{d-2k}\\ &= \sum_{k=0}^d \binom{d+k}{2k} \binom{2k}{k}
\left(\frac{x-1}{2}\right)^k
\end{split}
\end{equation*}
are the Legendre
polynomials (cf. \cite[p.~66]{Riordan:68}).
\end{theorem}
\begin{proof}

In a first step we solve the double integral in identity (\ref{vd}) for
$v_d^{(1)}$. Let $j\geq0$, $k\geq0$. Then
\begin{equation*}
\begin{split}
\int_{z=0}^{1}\!\int_{y=-2\sqrt{z}}^{2\sqrt{z}} y^j&(y+z+1)^k\:dy\:dz
= \int_{y=-2}^{2}\int_{z={y^2}/{4}}^{1} y^j(y+z+1)^k\:dz\:dy\\
&= \frac{1}{k+1}\Bigg(\int_{-2}^{2}y^j(y+2)^{k+1}\:dy-\int_{-2}^{2}y^j(y/2+1)^{2k+2}\:dy\Bigg)\\
&=\frac{1}{k+1}\Bigg(2^{j+k+2}\int_{-1}^{1}y^j(y+1)^{k+1}\:dy-2^{j+1}\int_{-1}^{1}y^j(y+1)^{2k+2}\:dy\Bigg)\\
\end{split}
\end{equation*}
where we performed the substitution $y/2 \rightarrow y$ in the last
step. By iterated partial integration we gain now from the last
expression that
\begin{equation}\label{partint}
\begin{split}
\int_{z=0}^{1}\!\int_{y=-2\sqrt{z}}^{2\sqrt{z}} y^j(y+z+1)^k\:dy\:dz
&= \frac{2^{j+2k+4}}{k+1}\Bigg(\sum_{r=1}^{j+1}
\frac{(-2)^{r-1}(j)_{r-1}}{(k+r+1)_{r}}-
\sum_{r=1}^{j+1} \frac{(-2)^{r-1}(j)_{r-1}}{(2k+r+2)_{r}}\Bigg)\\
\end{split}
\end{equation}
with $(x)_j \ce \prod_{i=0}^{j-1}(x-i)$.

In the following we insert (\ref{partint}) in formula
(\ref{vd})  and perform stepwise a first evaluation of
${v_d^{(1)}}/{v_d^{(0)}}$ mainly as a sum of products of
factorials.
\begin{equation*}
\begin{split}
\frac{v_d^{(1)}}{v_d^{(0)}}
=& \Bigg(\vphantom{\frac{2^{d(d + 1)/2}}{d!}\frac{1}{\prod_{i=0}^{d-1} \binom{2i+1}{i}}}2^{(d - 1)(d - 2)/2-2}
\sum_{j=0}^{d-2}\sum_{k=0}^{d-2-j}
\frac{(-1)^{d+k}2^{2d-2-2k-j}}{j!k!(d-2-j-k)!}
\prod_{i=1}^{d-2-k-j}\frac{2 + (d-2 - i - 1)/2}{3 + (2(d-2) - i - 1)/2}\\
&\frac{
\prod_{i=1}^{d-2-k}(1 + (d-2 - i)/2)
\prod_{i=1}^{d-2-k-j}(1 + (d-2 - i)/2)
}{
\prod_{i=1}^{d-2-k+d-2-k-j}(2 + (2(d-2) - i - 1)/2)
}
\frac{1}{\prod_{i=0}^{d-2-1} \binom{2i+1}{i}}\\
&
\frac{2^{j+2k+4}}{k+1}\Bigg(\sum_{r=1}^{j+1} \frac{(-2)^{r-1}(j)_{r-1}}{(k+r+1)_{r}}-
\sum_{r=1}^{j+1} \frac{(-2)^{r-1}(j)_{r-1}}{(2k+r+2)_{r}}\Bigg)
\vphantom{\frac{2^{d(d + 1)/2}}{d!}\frac{1}{\prod_{i=0}^{d-1}
\binom{2i+1}{i}}}\Bigg)/\Bigg(\frac{2^{d(d + 1)/2}}{d!}\frac{1}{\prod_{i=0}^{d-1} \binom{2i+1}{i}}\Bigg)\\
=&
\sum_{j=0}^{d-2}\sum_{k=0}^{d-j-2}
(-1)^{d+k+1}\frac{d!}{j!(k+1)!(d-j-k-2)!}
\prod_{i=1}^{d-j-k-2}\frac{d - i + 1}{2 d - i + 1}\\
&
\frac{
\prod_{i=1}^{d-k-2}(d - i)
\prod_{i=1}^{d-j-k-2}(d - i)
}{
\prod_{i=1}^{2d-j-2k-4}(2d - i - 1)
}
\frac{\prod_{i=0}^{d-1} \binom{2i+1}{i}}{\prod_{i=0}^{d-3} \binom{2i+1}{i}}
\Bigg(
\sum_{r=1}^{j+1} \frac{(-2)^{r}(j)_{r-1}}{(k+r+1)_{r}}-
\sum_{r=1}^{j+1} \frac{(-2)^{r}(j)_{r-1}}{(2k+r+2)_{r}}
\Bigg)\\
=& \sum_{j=0}^{d-2}\sum_{k=0}^{d-j-2}
(-1)^{d+k+1}\frac{d!}{j!(k+1)!(d-j-k-2)!} \frac{\frac{d!}{(j + k +
2)!}}{\frac{(2 d)!}{(d + j + k + 2)!}} \frac{ \frac{(d - 1)!}{(k +
1)!} \frac{(d - 1)!}{(j + k + 1)!} }{ \frac{(2 d - 2)!}{(j + 2 k +
2)!}
}\\
& \frac{(2d-3)!}{(d-2)!(d-1)!}\frac{(2d-1)!}{(d-1)!d!} \Bigg(
\sum_{r=1}^{j+1}
\frac{(-2)^{r}\frac{j!}{(j-r+1)!}}{\frac{(k+r+1)!}{(k+1)!}}-
\sum_{r=1}^{j+1}
\frac{(-2)^{r}\frac{j!}{(j-r+1)!}}{\frac{(2k+r+2)!}{(2k+2)!}}
\Bigg)\\
=& \sum_{j=0}^{d-2}\sum_{k=0}^{d-j-2} (-1)^{d+k+1}
\frac{(d + j + k + 2)!(j + 2 k + 2)!}{(d-j-k-2)!j!(j + k + 2)!(j + k + 1)!(k + 1)!(k + 1)!}\\
& \Bigg( \sum_{r=1}^{j+1}
\frac{(-2)^{r-2}j!(k+1)!}{(j-r+1)!(k+r+1)!}- \sum_{r=1}^{j+1}
\frac{(-2)^{r-2}j!(2k+2)!}{(j-r+1)!(2k+r+2)!}
\Bigg).\\
\end{split}
\end{equation*}
In the next step we rewrite the last expression as a sum over
products of binomial coefficients.
\begin{equation*}
\begin{split}
\frac{v_d^{(1)}}{v_d^{(0)}}
 =&
\sum_{j=0}^{d-2}\sum_{k=0}^{d-j-2}
(-1)^{d+k+1}
\binom{d}{j+k+2}
\binom{d+j+k+2}{d}
\frac{j+k+2}{j+2k+3}\\
& \Bigg( \sum_{r=1}^{j+1}
(-2)^{r-2}\binom{j+2k+3}{2k+r+2}\binom{2k+r+2}{k+1}-
\sum_{r=1}^{j+1} (-2)^{r-2}\binom{j+2k+3}{2k+r+2}\binom{2k+2}{k+1}
\Bigg).\\
\end{split}
\end{equation*}
Using the substitution $j+k+2 \rightarrow a, k+1 \rightarrow b$ the
latter expression reads
\begin{equation*}
\begin{split}
 &\sum_{a=2}^{d}\sum_{b=0}^{a-1}\sum_{r=1}^{a-b} (-1)^{d+b}
(-2)^{r-2} \frac{a}{a+b} \binom{d}{a} \binom{d+a}{d}
\binom{a+b}{2b+r} \binom{2b+r}{b}-
\\
& \sum_{a=2}^{d}\sum_{b=0}^{a-1}\sum_{r=1}^{a-b} (-1)^{d+b}
(-2)^{r-2} \frac{a}{a+b} \binom{d}{a} \binom{d+a}{d}
\binom{a+b}{2b+r} \binom{2b}{b}\\
\end{split}
\end{equation*}
so that
\begin{equation}\label{twosumsdifference}
\begin{split}
\frac{v_d^{(1)}}{v_d^{(0)}} =& \sum_{a=2}^{d} (-1)^{d} a
\binom{d}{a} \binom{d+a}{d} \sum_{r=1}^{a} (-2)^{r-2}
\\
&
\Bigg(
\sum_{b=0}^{a-r}
(-1)^{b}
\frac{1}{a+b}
\binom{a+b}{2b+r}
\binom{2b+r}{b}-
\sum_{b=0}^{a-r}
(-1)^{b}
\frac{1}{a+b}
\binom{a+b}{2b+r}
\binom{2b}{b}
\Bigg).\\
\end{split}
\end{equation}
In the following we will simplify the two innermost sums.

We start with the first sum.
If $r=a$ the sum trivially equals $\frac 1a$. Let us assume $1\le
r\le a-1$ now. Then we have
\begin{equation*}
\begin{split}
\sum_{b=0}^{a-r} (-1)^{b} \frac{1}{a+b} \binom{a+b}{2b+r}
\binom{2b+r}{b} =& \frac{1}{a-r} \sum_{b=0}^{a-r} (-1)^{b}
\binom{a-r}{b} \binom{a+b-1}{b+r}\\
=& \frac{(-1)^r}{a-r} \sum_{b=0}^{a-r} \binom{a-r}{b}
\binom{r-a}{b+r} = \frac{(-1)^r}{a-r} \binom{0}{a} = 0,
\end{split}
\end{equation*}
where we used
\begin{equation*}
(-1)^k\binom{k-n-1}{k} =\binom{n}{k} \qquad (n \in \Z, k \ge 0)
\end{equation*}
for the second identity, and Vandermonde's identity
\begin{equation*}
\begin{split}
\sum_{k=0}^n\binom{n}{k}\binom{s}{k+t} =&
\sum_{k=0}^n\binom{n}{k}\binom{s}{n+t-k}\\
=& \binom{n+s}{n+t}
\qquad (s \in \Z, n,t \ge 0)\\
\end{split}
\end{equation*}
for the third one.

Altogether we have established
\begin{equation}\label{firstsum}
\begin{split}
\sum_{b=0}^{a-r} (-1)^{b} \frac{1}{a+b} \binom{a+b}{2b+r}
\binom{2b+r}{b}= \frac{1}{a}\delta_{r,a}  \qquad (1\le r\le a),
\end{split}
\end{equation}
where $\delta_{r,a}$ denotes the Kronecker symbol.

Now we turn to the second sum in question. Since this is a sum reminiscent of a sum treated in \cite[Section 5.2,
Problem~7]{Knuth:94} we first try to adopt the strategy followed
there and use \cite[Section 5.1, identity 5.26]{Knuth:94}
\begin{equation}
\begin{split}
\binom{\l +q+1}{m+n+1}=\sum_{0\le k \le \l} \binom{\l -k}{m}
\binom{q+k}{n}
 \quad (\l,  m \ge 0, n\ge q\ge 0).
\end{split}
\end{equation}
With
$\l = a+b-1, q=0, m=2b, n=r-1$ and $k=s$ we get
\begin{equation*}
\begin{split}
\sum_{b=0}^{a-r} (-1)^{b} \frac{1}{a+b} \binom{a+b}{2b+r}
\binom{2b}{b} =& \sum_{b=0}^{a-r}\sum_{s=0}^{a+b-1}
\frac{(-1)^{b}}{a+b} \binom{a+b-s-1}{2b} \binom{s}{r-1}
\binom{2b}{b}
\end{split}
\end{equation*}
which by a change of summations yields
\begin{equation}\label{changesums}
\begin{split}
=& \sum_{s=r-1}^{2a-r-1} \binom{s}{r-1} \sum_{b=0}^{a-s-1}
\frac{(-1)^{b}}{a+b} \binom{a+b-s-1}{2b}
\binom{2b}{b}\\
=& \sum_{s=r-1}^{a-1} \binom{s}{r-1} \sum_{b=0}^{a-s-1}
\frac{(-1)^{b}}{a+b} \binom{a+b-s-1}{2b} \binom{2b}{b}.
\end{split}
\end{equation}
Now we are ready to apply sum $S_m$ from \cite[Section 5.2,
Problem~8]{Knuth:94}
\begin{equation}
\begin{split}
S_m=\sum_{k=0}^n(-1)^k\frac{1}{k+m+1}\binom{n+k}{2k}\binom{2k}{k}=&(-1)^n\frac{m!n!}{(m+n+1)!}\binom{m}{n}
\qquad (m,n \ge 0).
\end{split}
\end{equation}
With $m=a-1, n=a-s-1$ and $k=b$ we find that (\ref{changesums}) from
above equals
\begin{equation}\label{equalsum}
\begin{split}
&\sum_{s=r-1}^{a-1} \binom{s}{r-1}
\frac{(-1)^{a+s+1}(a-1)!(a-s-1)!}{(2a-s-1)!}
\binom{a-1}{a-s-1}\\
=& \frac{(-1)^{a+1}(a-1)!(a-1)!}{(2a-1)!} \binom{2a-1}{r-1}
\sum_{s=r-1}^{a-1} (-1)^s
\binom{2a-r}{s-r+1}\\
 =&
\frac{(-1)^{a+r}(a-1)!(a-1)!}{(2a-1)!} \binom{2a-1}{r-1}
\sum_{s=0}^{a-r} (-1)^s
\binom{2a-r}{s}\\
\end{split}
\end{equation}
(where we applied the substitution $s-r+1 \rightarrow s $ in the
last step). Using the basic identity
\begin{equation*}
\begin{split}
\sum_{j=0}^k(-1)^j\binom{n}{j}=&(-1)^k\binom{n-1}{k}
\qquad (n,k \ge 0)\\
\end{split}
\end{equation*}
to evaluate the last sum in (\ref{equalsum}) we finally get
\begin{equation}\label{secondsum}
\begin{split}
\sum_{b=0}^{a-r} (-1)^{b} \frac{1}{a+b} \binom{a+b}{2b+r}
\binom{2b}{b}
 =&\frac{(a-1)!(a-1)!}{(2a-1)!} \binom{2a-1}{r-1}
\binom{2a-r-1}{a-r}\\
=& \frac{1}{2a-r}
\binom{a-1}{a-r}.\\
\end{split}
\end{equation}

Now we go on plugging the results (\ref{firstsum}) and
(\ref{secondsum}) from above in (\ref{twosumsdifference}) and find
\begin{equation}
\begin{split}
\frac{v_d^{(1)}}{v_d^{(0)}} =& \sum_{a=2}^{d} (-1)^{d} a
\binom{d}{a} \binom{d+a}{d} \Bigg( \frac{(-2)^{a-2}}{a} -
\sum_{r=1}^{a} (-2)^{r-2} \frac{1}{2a-r} \binom{a-1}{a-r}
\Bigg)\\
 \quad =& \sum_{a=2}^{d} (-1)^{d+1} a
\binom{d}{a} \binom{d+a}{d} \Bigg( \sum_{r=0}^{a-1} (-2)^{r-1}
\frac{1}{2a-r-1} \binom{a-1}{r} - (-2)^{a-2}\frac{1}{a} \Bigg).
\\
\end{split}
\end{equation}
In order to get rid of the inner sum we use an identity that may be proved
as an application of the classical reflection law
\begin{equation}
\frac1{(1-z)^a}F\left( {{a,b} \above 0pt c} \left| {\frac{-z}{1-z}} \right. \right)=
F\left(\left. {{a,c-b} \above 0pt c} \right| z  \right)
\end{equation}
for hypergeometric
functions by J.F. Pfaff \cite{Pfaff:1797}, namely
\begin{equation}
\begin{split}
\sum_{k=0}^m (-2)^k \frac{2m+1}{2m-k+1} \binom{m}{k} =&
\frac{(-1)^m2^{2m}}{\binom{2m}{m}}
\qquad (m \ge 0),\\
\end{split}
\end{equation}
compare \cite[identity (5.104)]{Knuth:94}. In this way we find

\begin{equation}\label{integer}
\begin{split}
\frac{v_d^{(1)}}{v_d^{(0)}} =& \sum_{a=2}^{d} (-1)^{d+a+1} a
\binom{d}{a} \binom{d+a}{d} \Bigg( 2^{2a-3} \frac{1}{2a-1}
\frac{1}{\binom{2a-2}{a-1}} -2^{a-2} \frac{1}{a} \Bigg)
\\
=&
\sum_{a=2}^{d}
(-1)^{d+a}
\binom{d}{a}
\binom{d+a}{d}
\Bigg(
2^{a-2}-
2^{2a-2}
\frac{1}{\binom{2a}{a}}
\Bigg)
\\
=& \sum_{a=2}^{d} (-1)^{d+a} 2^{a-2} \binom{d+a}{2a} \Bigg(
\binom{2a}{a} -2^a \Bigg),
\end{split}
\end{equation}
so that we have proved that the ratio ${v_d^{(1)}}/{v_d^{(0)}}$
is an integer.

In the last step of the proof we establish the explicit formula for the ratios. Recall
 the Legendre polynomials $P_d(x)$, as defined in the theorem,  and let
\begin{equation}
\begin{split}
\rho_d(x) \ce& \sum_{k=0}^d \binom{d+k}{d-k} x^k
\end{split}
\end{equation}
denote the associated Legendre
polynomials (cf. \cite[p.~66]{Riordan:68}). Then  (\ref{integer})
yields
\begin{equation}
\begin{split}
\frac{v_d^{(1)}}{v_d^{(0)}} =& (-1)^d\frac{P_d(-3)-\rho_d(-4)}{4}.
\\
\end{split}
\end{equation}
Now (cf. \cite[p.~158]{Rainville:60})
\begin{equation}
\begin{split}
P_d(-x)=(-1)^d P_d(x).
\end{split}
\end{equation}
Furthermore $\rho_d$ satisfies the recursive formula
\begin{equation}
\begin{split}
\rho_{d}(x)=&(x+2)\rho_{d-1}(x)-\rho_{d-2}(x)\\
\rho_0(x)=&0, \rho_1(x)=x+1
\end{split}
\end{equation}
(compare  \cite[p.~66]{Riordan:68}) so that we have
\begin{equation*}
\begin{split}
(-1)^d\rho_d(-4)=2d+1
\end{split}
\end{equation*}
which completes the proof of  (\ref{mainresult}).

\end{proof}

\section{Recurrence,  asymptotic behaviour and probabilities}

In this section we apply Theorem \ref{mainresult} in order to establish a recurrence for the
quotients $v_d^{(1)}/{v_d^{(0)}}$  as well as to establish the asymptotic behaviour of this
sequence for $d\to \infty$ and its consequence on the probabilities $v_d^{(1)}/{v_d}.$

Since the Legendre polynomials satisfy the recursive formula
\begin{equation}
\begin{split}
dP_{d}(x)-&(2d-1)xP_{d-1}(x)+(d-1)P_{d-2}(x)=0\\
P_0(x)=&1, P_1(x)=x
\end{split}
\end{equation}
(cf. \cite[p.~160]{Rainville:60}) we get the following second
order linear recurrence for our ratios ${v_d^{(1)}}/{v_d^{(0)}}$.

\begin{corollary}
\begin{equation}
\begin{split}
d\frac{v_d^{(1)}}{v_d^{(0)}}-3(2d-1)&\frac{v_{d-1}^{(1)}}{v_{d-1}^{(0)}}+(d-1)\frac{v_{d-2}^{(1)}}{v_{d-2}^{(0)}}=2d(d-1),
\qquad (d\ge 2)\\
 \frac{v_0^{(1)}}{v_0^{(0)}}=&\frac{v_1^{(1)}}{v_1^{(0)}}= 0.
\end{split}
\end{equation}
\end{corollary}
We turn our attention now to the asymptotic behaviour of the ratios
for $d\to\infty$ and start by their generating function.
The generating function of the Legendre polynomials is given by
(\cite[p.~78]{Riordan:68})

\begin{equation}
\begin{split}
\sum_{d \ge 0} P_d(x)z^d = \frac 1{\sqrt {1-2xz+z^2}},
\end{split}
\end{equation}

so that the generating function of our ratios reads
\begin{corollary}
\begin{equation}
\begin{split}
V_1(z):=\sum_{d \ge 0}\frac{v_d^{(1)}}{v_d^{(0)}} z^d = \frac 14 \left(\frac
1{\sqrt{1-6z+z^2}}-\frac {1+z}{(1-z)^2}\right).
\end{split}
\end{equation}
\end{corollary}

Performing singularity analysis the latter result allows  to
establish the asymptotic behaviour of the ratios for $d\to \infty$
as follows.
\begin{proposition}\label{asympt} For $d\to \infty$
\begin{equation}
\begin{split}
\frac{v_d^{(1)}}{v_d^{(0)}}= \frac 1{8\root 4 \of 2\sqrt{\pi d}}(3+2\sqrt 2)^{d+\frac 12}\left(1+\mathcal{O}\left(\frac 1d\right)\right)    .
\end{split}
\end{equation}
\end{proposition}
\begin{proof} We adopt the usual technique of singularity analysis of generating functions,
compare e.g. \cite[Chapter IV] {FlajoletSedgewick:09} or \cite[Chapter 8]{Szpankowski:01}. The dominating singularity of the generating function $V_1(z)$ is given by the zero  $3-2\sqrt 2$ of $1-6z+z^2$ closest to the origin, whereas the other zero of $1-6z+z^2$ as well as the term $\frac {1+z}{(1-z)^2}$ will give a contribution that is exponentially smaller than the contribution of the main term. The local expansion of $V_1(z)$ about the dominating singularity  reads
\begin{equation*}
\begin{split}
V_1(z)=\frac 1{8\root 4 \of 2 \sqrt {3-2\sqrt 2}}\left( 1 - \frac z{3-2\sqrt 2}\right)^{-1/2}\left(1+\mathcal{O} \left( 1 - \frac z{3-2\sqrt 2}\right)\right) \text{for } z\to 3-2\sqrt 2, \end{split}
\end{equation*}
from which the asymptotics is immediate.
\end{proof}

In \cite{API} Akiyama and Peth\H{o} also discussed the probabilities
\begin{equation}
p_d^{(s)}\ce{v_d^{(s)}}/{v_d}
\end{equation}
for a contractive polynomial  of degree $d$ in $\R[x]$ to have $s$ pairs of complex conjugate roots. In particular they derived (cf. \cite[Theorem 6.1]{API})
\begin{equation}\label{pd0}
\begin{split}
\log p_d^{(0)} = -\frac {\log 2}2 d^2 + \frac 18 \log d + \mathcal{O} (1), \quad \text {for } d\to \infty,
\end{split}
\end{equation}
for the probability of totally real polynomials and, by numerical evidence for $d\le 100,$ conjectured that
\begin{equation}
\begin{split}
\log p_d^{(1)} \le -\frac {\log 2}2 d^2 + d \log q
\end{split}
\end{equation}
for some constant $q$. Now, obviously,
$$p_d^{(1)}= \frac{v_d^{(1)}}{v_d^{(0)}}p_d^{(0)},$$
so that from (\ref{pd0}) and our Proposition \ref{asympt} we gain
\begin{corollary}
The probability $p_d^{(1)}$ for a contractive polynomial  of degree $d$ in $\R[x]$ to have exactly one pair of complex conjugate roots fulfills
\begin{equation}
\begin{split}
\log p_d^{(1)} = -\frac {\log 2}2 d^2 + d \log (3+2\sqrt 2) + \mathcal{O}(\log d) \quad \text{for } d\to \infty.
\end{split}
\end{equation}
\end{corollary}

\section{Concluding remarks}
In this paper we were able to settle the instance $s=1$ of Conjecture \ref{conjecture}. The question arises, whether our methods could be used to prove the conjecture for additional instances of $s\ge 2$ or even for general $s\ge 1.$ A crucial point for a possible application of our method would be to establish a generalization of the Selberg-Aomoto integral for integrands that will occur
with the evaluation of  ${v_d^{(s)}}$ for $s\ge 2$ similar to formula $(\ref{vd})$ in the instance $s=1$. Work is in
progress on this question, but even the explicit evaluation of the integrals that appear in instance $s=2$
seems to be very hard.

\bibliographystyle{siam}

\bibliography{KirschenhoferWeitzer}
\end{document}